\documentclass{amsart}
\usepackage{amsthm}
\usepackage{csquotes}
\usepackage{mathrsfs}
\usepackage{IEEEtrantools}
\newtheorem{theorem}{Theorem}[section]

\theoremstyle{definition}
\newtheorem{definition}[theorem]{Definition}

\newtheorem{corollary}[theorem]{Corollary}
\newtheorem{example}[theorem]{Example}

\theoremstyle{remark}

\numberwithin{equation}{section}

\newcommand{\implication}[2]{\mbox{\enquote{#1$\implies$#2}}}

\newcommand{\RomanNumeralCaps}[1]
{\MakeUppercase{\romannumeral #1}}

\begin{document}

\title[Pluricomplex Green functions on Stein manifolds]{Pluricomplex Green Functons on Stein manifolds \\ and certain linear topological invariants}

\author{Ayd\i n Aytuna}
\address{Institute of Applied Mathematics, Middle East Technical University, 06800 \c Cankaya, Ankara, Turkey}
\curraddr{Institute of Applied Mathematics, Middle East Technical University, 06800 \c Cankaya, Ankara, Turkey}
\email{aytuna@metu.edu.tr}

\subjclass[2020]{Primary 32A70, 32U35; Secondary 46E10}

\date{June 06, 2022.}

\dedicatory{Sunar'a}

\keywords{Pluricomplex Green functions, Fr\'echet spaces of analytic functions, diametral dimension, $\widetilde{\Omega}$ condition, semi-proper negative plurisubharmonic functions}

\begin{abstract}
In this paper we explore the existence of pluricomplex Green functions for
Stein manifolds from a functional analysis point of view. For a Stein
manifold $M$, we will denote by $O\left( M\right) $ the Fr\'{e}chet space of
analytic functions on $M$ equipped with the topology of uniform convergence
on compact subsets. In the first section, we examine the relationship
between existence of pluricomplex Green functions and the diametral
dimension of  $O\left( M\right) .$ This led us to consider negative
plurisubharmonic functions on $M$ with a nontrivial relatively compact
sublevel set (semi-proper). In section 2, we characterize Stein manifolds
possessing a semi-proper negative plurisubharmonic function through a local,
controlled approximation type condition, which can be considered as a local
version of the linear topological invariant $\widetilde{\Omega }$ of D. Vogt.
In Section 3 we look into pluri-Greenian and locally uniformly pluri-Greenian
complex manifolds introduced by E. Poletsky. We show that a complex manifold
is locally uniformly pluri-Greenian if and only if it is pluri-Greenian and
give a characterization of locally uniformly pluri-Greenian Stein manifolds
in terms of the notions introduced in Section 2.    
\end{abstract}
\maketitle

\section*{Introduction}\label{intro}
Spaces of analytic functions form an essential class of nuclear Fr\'{e}chet spaces. In recent years there have been significant advances in the structure theory of nuclear Fr\'{e}chet spaces  (see \cite{MV}). In attempts to utilise this theory in complex analysis,
one is led to analyse the complex analytic properties shared by Stein manifolds whose
analytic function spaces possess a common linear topological invariant.

In this note, we look into the connection between the existence of
pluricomplex Green functions and certain linear topological properties of
the Fr\'{e}chet space of analytic functions on a Stein manifold; a topic
that was touched upon in \cite{Ay}.

Throughout the note we will denote the space of analytic functions on a
Stein manifold $M$ with compact open topology by $\mathcal{O}(M)$.

After compiling some background material, in section \ref{section1} we
recall a result of \cite{Ay} that relates the existence of pluricomplex Green functions on $M$ to the diametral dimension of $\mathcal{O}(M)$ and look at an example.

In section \ref{section2}, we examine semi-proper pluricomplex Green
functions (see definiton below) and give a characterization of their
existence in terms of a controlled approximation property for analytic
functions on the complex manifold.

We then relate the existence of semi-proper pluricomplex Green functions to
the linear topological properties $\widetilde{\Omega}$ of Vogt \cite{VB}, and the
diametral dimension for $\mathcal{O}(M)$.

In section \ref{section3} we look into pluri-Greenian and locally uniformly pluri-Greenian
complex manifolds introduced by E. Poletsky \cite{P paper}, \cite{P yeni}. We show that a complex manifold is locally uniformly
pluri-Greenian if and only if it is pluri-Greenian and give a
characterization of locally uniformly pluri-Greenian Stein manifolds in
terms of the notions introduced in section 2. 

In the last section, we recall some research done for the existence of
pluricomplex Green functions, with the intention of
putting our results in perspective.

The manifolds considered in this note are always assumed to be connected. We
will use the standard terminology and results from functional analysis and
complex potential theory as presented in \cite{MV,K}, respectively.

\section{Pluricomplex Green functions on a Stein manifold $M$ and the diametral dimension of $\mathcal{O}(M)$}\label{section1}

We start by defining the terms appearing in the heading. 

Let $\mathcal{X}$ be a Fr\'echet space and $\lbrace \parallel\cdot\parallel_k\rbrace^\infty_k$ an increasing system of semi-norms generating the topology of $\mathcal{X}$. Let $\mathcal{U}_k:=\lbrace x\in \mathcal{X}:\parallel x\parallel_k\leq 1\rbrace$ denote the unit ball of the corresponding local Banach spaces $\mathcal{X}_k$, $k=1,2,\cdots$. For a subspace $L$ of $\mathcal{X}$ and for $k^+>k$, let;
$$\delta(\mathcal{U}_{k^+};\mathcal{U}_k,L):= \inf\lbrace t>0: \mathcal{U}_{k^+}\subseteq t\mathcal{U}_k+L\rbrace.$$
The \emph{$n$-th Kolmogorov diameter} of $\mathcal{U}_{k^+}$ with respect to $\mathcal{U}_k$ is defined as:
$$d_n(\mathcal{U}_{k^+},\mathcal{U}_k):= \inf\lbrace \delta(\mathcal{U}_{k^+};\mathcal{U}_k,L): L \textrm{ subspace of $\mathcal{X}$ of dimension at most } n, n\in \mathbb{Z}^+\rbrace.$$
There is a huge literature on techniques of computing Kolmogorov diameters for specific Banach spaces (see \cite{Pinkus}).

The \emph{diametral dimension} of $\mathcal{X}$ is defined as:
$$\Delta (\mathcal{X}):= \lbrace (t_n)_n\in \mathbb{C}^{\mathbb{N}};\, \forall\, k, \, \exists\, k^+>k; \, t_nd_n(\mathcal{U}_{k^+},\mathcal{U}_k)\to 0\rbrace.$$
The diametral dimension of $\mathcal{X}$ does not depend upon the generating semi-norm system and is a linear topological invariant of the Fr\'echet space $\mathcal{X}$ in the sense that if $\mathcal{X}$ is linear topologically isomorphic to a Fr\'echet space $\mathcal{Y}$, then $\Delta(\mathcal{X})=\Delta(\mathcal{Y})$. 

In the case of nuclear Fr\'echet spaces, one can choose a generating system of Hilbertian semi-norms $\lbrace \parallel\cdot\parallel_k\rbrace_k$ for $\mathcal{X}$ and represent the compact linking maps ${i^k_{k^+}:\mathcal{X}_{k^+}\to \mathcal{X}_k}$ as:
$$i_{k^+}^k:= \sum_n\lambda_n \langle \cdot,f_n\rangle_{k^+} e_n,$$
where $\lbrace f_n\rbrace$ and $\lbrace e_n\rbrace$ are orthonormal sequences in $\mathcal{X}_{k^+}$ and $\mathcal{X}_k$, respectively and $(\lambda_n)$ is a sequence of nonnegative numbers converging to $0$.

In this framework, the sequence of Kolmogorov diameters $\lbrace d_n(\mathcal{U}_{k^+},\mathcal{U}_k)\rbrace$ coincides with the non-increasing rearrangement of the sequence $\lbrace \lambda_n\rbrace_n$ of the eigenvalues. 

This information is quite useful in computing diametral dimensions. For example, in view of this observation, one can readily verify that 
\begin{align*}
	\Delta(\mathcal{O}(\Delta^d))&= \lbrace (\xi_n): \sum_n|\xi_n|^2e^{-\frac{2}{k}n^{1/d}}<\infty, \,\forall\, k=1,2,\cdots \rbrace\\
	\Delta (\mathcal{O}(\mathbb{C}^d))&= \lbrace (\xi_n): \,\exists\, R>1; \, |\xi_n|\leq CR^{n^{1/d}}\rbrace, 
\end{align*}
where $\Delta^d$ is the unit polydisc in the $d$ dimensional complex plane $\mathbb{C}^d$, $d=1,2,\cdots$. 

For more information about diametral dimension, we refer the reader to \cite{Pietsch,Rolewicz}. 

We now recall some background material from Complex Potential Theory.

Let $M$ be a connected complex manifold and $z_0\in M$. Let us denote plurisubharmonic functions on an open $\Omega\subseteq M$ by the symbol $PSH(\Omega)$. The \emph{pluricomplex Green function} $g_M(\cdot;z_0)$ with a pole at $w_0$ is defined as:
$$g_M(z;z_0):=\sup\lbrace u\in PSH(M):u<0\,\textrm{on}\,M, \exists\, c:\, u(z)\leq\ln|z-w_0|+c, z\in U_{z_0}\rbrace,$$
where $U_{z_0}$ is some local neighborhood near $z_0$ (see \cite{K}).

It turns out that either $g_M$ is a plurisubharmonic function that is maximal on $M$ or is $-\infty$. We will say that pluricomplex Green function with a pole at $z_0$ \emph{exists} in case $g_M(\cdot, z_0)\not\equiv-\infty$.

Pluricomplex Green functions are natural analogues in several variables, of the classical Green functions in the theory of open Riemann surfaces and were studied in diverse settings (see \cite{K,Demailly,P paper}).

Our first result, connecting pluricomplex Green functions and diametral dimension comes basically from \cite{Ay}. Indeed Corollary 4.3, Theorem 2.9 and Proposition 4.1 of \cite{Ay} combined yields:

\begin{theorem}\label{thm1}
	Let $M$ be a Stein manifold of dimension $d$, and suppose that $\Delta(\mathcal{O}(M))=\Delta(\mathcal{O}(\Delta^d))$. Then pluricomplex Green function exists for each point of $M$.
\end{theorem}

However the statement of the above theorem is not an if and only if statement as the example below reveals.

\begin{example}
	Let $M=\Delta\times \mathbb{C}\subseteq \mathbb{C}^2$. Clearly, pluricomplex Green function exists for each point $\zeta_0=(z_0,w_0)$ of $M$, indeed the function:
	$$g((z,w);\zeta_0)=\ln\big|\frac{z-z_0}{1-z\overline{z_0}}\big|$$
	is the pluricomplex Green function for $M$ with a pole at $\zeta_0$. 
	
	We now examine the diametral dimension of $\mathcal{O}(M)$. 
	
	The nuclear Fr\'echet space $\mathcal{O}(M)$ admits a Schauder basis $\lbrace \boldsymbol{z^n}\boldsymbol{w^m}\rbrace$, where $${\boldsymbol{z^n}(z,w)=z^n}, \quad {\boldsymbol{w^m}(z,w)=w^m}, \, n,m=0,1,\cdots.$$
	
	We choose the exhaustion $D_k=\lbrace (z,w)\in M: |z|\leq e^{-1/k},\, |w|\leq e^k\rbrace$ of $M$, $k=1,2\cdots$, and set $h_k(\theta):= (1-\theta)k -\frac{\theta}{k}$, for $0\leq \theta\leq 1$, $k=1,2,\cdots$. 
	
	Using the $H^2(D_k)$ norms, $k=1,2,\cdots$, to generate the topology of $\mathcal{O}(M)$, we obtain a sequence space representation of $\mathcal{O}(M)$ of the form:
	\begin{eqnarray*}
		\mathcal{O}(M)&\cong &\lbrace (\xi_{n,m})_{n,m=0}^\infty: |(\xi_{n,m})|_k\\
		& &\quad =\big(\sum_{\vec{\eta} =(n,m)} |\xi_{n,m}|^2e^{2|\vec{\eta}|h_k(\frac{n}{n+m})}\big)^{1/2}<\infty,\,\forall\, k=1,2,\cdots \rbrace,
	\end{eqnarray*}
	where, as usual, we used the notation $|\vec{\eta}|:=n+m$ for $\vec{\eta}=(n,m)\in \mathbb{N}\times \mathbb{N}$. 
	
	The corresponding fundamental system of neighborhoods are:
	$$\mathcal{U}_k=\lbrace f=\sum_{n,m}\xi_{n,m}\boldsymbol{z^nw^m}: |(\xi_{n,m})|_k\leq 1,\, k=1,2,\cdots \rbrace.$$
	Fix a $k\in \mathbb{N}$, $k^+ \gg 2k$ and $\varepsilon$, $0<\varepsilon <1$, whose values we will specify later. We will use the usual lexicographical order for $\mathbb{N}\times\mathbb{N}$, and set $\theta(\vec{\eta}):=\frac{n}{n+m}$ for $\overline{\eta}=(n,m)$, $n,m=0,1,\cdots$.
	
	As was mentioned in the introduction, the sequence of Kolmogorov diameters $\lbrace d_0(\mathcal{U}_{k^+},\mathcal{U}_k),\cdots,d_n(\mathcal{U}_{k^+},\mathcal{U}_k),\cdots\rbrace$ coincides with the non-increasing rearrangement of $$\lbrace e^{2|\vec{\eta}|(h_k(\theta(\vec{\eta})-h_{k^+}(\theta(\vec{\eta}))}\rbrace_{\vec{\eta}\in\mathbb{N}\times\mathbb{N}}.$$
	
	Evidently, for a fixed $j$;
	\begin{equation}\label{star1}
	j<j+1\leq \boldsymbol{\#}\lbrace \vec{\eta}:|\vec{\eta}|(h_{k^+}(\theta(\vec{\eta}))-h_k(\theta(\vec{\eta})))\leq \varepsilon_j(k^+;k)\rbrace, 
	\end{equation}
	where $\#(A)$ denotes the number of elements of the set $A$, and $\varepsilon_j(k^+;k)=-\ln d_j(\mathcal{U}_{k^+}; \mathcal{U}_k)$, $j=0,1,\cdots$. 
	
	On the lattice points $\vec{\eta}$; $\theta(\vec{\eta})\leq 1-\varepsilon$, since $h_{k^+}(\theta(\vec{\eta}))-h_k(\theta(\vec{\eta}))\geq \varepsilon (k^+-k)$, we have:
	\begin{eqnarray*}
	\lefteqn{\#\{\vec{\eta}: \theta(\vec{\eta})\leq 1-\varepsilon,\, |\vec{\eta}|(h_{k^+}(\theta(\vec{\eta}))-h_k(\theta(\vec{\eta})))\leq \varepsilon_j(k^+;k)\}}\\
		&\leq & \# \{\vec{\eta}:|\vec{\eta}|\leq \frac{1}{\varepsilon (k^+-k)}\varepsilon_j(k^+;k)\}\\
		&\leq & \frac{1}{2} \{\big(\frac{1}{\varepsilon(k^+-k)}\varepsilon_j(k^+;k)\big)^2+\frac{1}{\varepsilon(k^+-k)}\varepsilon_j(k^+;k)\}\\
		&\leq &K\varepsilon_j(k^+;k)^2,
	\end{eqnarray*}

	where $K=\frac{1}{2}\big(\frac{1}{\varepsilon^2(k^+-k)}+\frac{1}{\varepsilon(k^+-k)}\big)$.

	On the remaining lattice points $\vec{\eta}$; $\theta(\vec{\eta})>1-\varepsilon$, since $k^+\gg 2k$, we estimate:
	\begin{eqnarray*}
	\lefteqn{\#\{\vec{\eta}:\theta(\vec{\eta})>1-\varepsilon, |\vec{\eta}|(h_{k^+}(\theta(\vec{\eta}))-h_k(\theta(\vec{\eta})))\leq \varepsilon_j(k^+;k)\}}\\
		&\leq& \#\{\vec{\eta}:\theta(\vec{\eta})>1-\varepsilon ; |\vec{\eta}|\leq 2k\varepsilon_j(k^+;k)\}.	
	\end{eqnarray*}
	A crude estimate on the number of lattice points in the wedge with base an interval in the unit simplex starting from $(0,1)$ with length $\sqrt{2}\varepsilon $ and with norm less than or equal to $2k\varepsilon_j(k^+;k)$ yields:
	$$\#\{\vec{\eta}: \theta(\vec{\eta})>1-\varepsilon ;|\vec{\eta}|\leq 2k \varepsilon_j(k^+;k)\}\leq 1+\frac{\varepsilon}{2}(4k^2+2k):=K'.$$
	Now choose $\varepsilon$ so that $K'\leq 2$ and $k^+$ so that $K<1$. It follows from these choices and inequality (\ref{star1}) that
	$$j\leq 3 \varepsilon_j(k^+;k)^2=3\ln \big(\frac{1}{d_j(k^+,k)}\big)^2,\quad \textrm{or}$$
	\begin{equation}\label{star2}
	d_j(\mathcal{U}_{k^+},\mathcal{U}_k)\big(e^{\frac{1}{\sqrt{3}}\sqrt{j}}\big)_j\leq 1.
	\end{equation}
	Since $j$ was arbitrary, the nuclearity of $\mathcal{O}(M)$ implies that $\big(e^{\frac{1}{\sqrt{3}}{\sqrt{j}}}\big)_j\in \Delta(\mathcal{O}(M))$ (see \cite{Frerick et al}). On the other hand, since $\Delta(\mathcal{O}(\Delta^2))=\{(\xi_j):\sum_j|\xi_j|^2e^{-\frac{2\sqrt{j}}{k}}<\infty, \, \forall\, k=1,2,\cdots\}$, $\big(e^{\frac{1}{\sqrt{3}}{\sqrt{j}}}\big)_j\notin \Delta(\mathcal{O}(\Delta^2))$.
\end{example}
For more information about the diametral dimension of spaces of analytic functions on complete Reinhardt domains in $\mathbb{C}^n$, we refer the reader to \cite{AKT manuscripta,AT Decomposition}.
\section{Semi-proper pluricomplex Green functions}\label{section2}

The sublevel sets of the pluricomplex Green functions of the above example are far from being relatively compact. Moreover the poles of these
	pluricomplex Green functions are not strict in the sense of Poletsky. Recall
	that a function $g$ on a complex manifold is said to have a \textit{%
		strict} \textit{logarithmic} \textit{pole} at a point $w$, in 
case there are a coordinate neighborhood $U$ and constants 
$c_{1}$ and $c_{2}$ such that  
\begin{equation*}
\ln \left\Vert z-w\right\Vert +c_{1}\leq g\left( w\right) \leq \ln
\left\Vert z-w\right\Vert +c_{2}
\end{equation*}%
on $U$ (see \cite{P paper}). One can ask
as to whether these features account for the difference; $%
\Delta (\mathcal{O}(M))\neq \Delta (\mathcal{O}(\Delta ^{2}))$ in the
example above.

In what follows, we will call an extended real valued function $f$ on a topological space $X$ \emph{semi-proper} in case there exists a $c\in \mathbb{R}$ such that the sublevel set $\emptyset\neq \{x\in X:f(x)<c\}$ is relatively compact in $X$.

To answer the question posed above we will enquire into the existence of semi-proper pluricomplex Green functions. 
\begin{definition}\label{defn1}
	Let $M$ be a Stein manifold and $\{\parallel\cdot\parallel_k\}_k$ be a generating norm system for $\mathcal{O}(M)$. For $z\in M$ we will say that $M$ has the property $\widetilde{\Omega}_z$ in case there exists a ball with radius $\varepsilon>0$ in some coordinate system around $z$, $B(z;\varepsilon)$, $B(z;\varepsilon)\subset\subset M$ such that
	$$\exists\, k_1,\lambda>0;\,\forall\, k>k_1, \exists\, C>0; \quad \mathcal{U}_{k_1}\subseteq \frac{C}{r}\mathcal{U}_0+r^\lambda \mathcal{U}_k,\, \forall\, r>0,$$
	where $\mathcal{U}_{0}:=\{f\in \mathcal{O}(M):\sup_{w\in B(z;\varepsilon)}|f(w)|\leq 1\}$. 
\end{definition}

Clearly, this condition does not depend on the generating norm system of $\mathcal{O}(M)$. 

Now we can state the main theorem of this section.

\begin{theorem}\label{thm2}
	Let $M$ be a Stein manifold and $\zeta_0\in M$. The following statements are equivalent:
	\begin{enumerate}
		\item There exists a negative plurisubharmonic function on $M$ that has a relatively compact sublevel set containing $\zeta_0$. 
		\item $M$ possesses a semi-proper pluricomplex Green function with a strict logarithmic pole at $\zeta_0$.
		\item $M$ has the property $\widetilde{\Omega}_{\zeta_0}$.
	\end{enumerate}
\end{theorem}

\begin{proof}
\implication{(1)}{(2)}	The proof of this implication is a gluing argument (see e.g. \cite{A D-issue}). We give it in detail since we will utilize the explicit form of some of the functions appearing in the proof in a later discussion.

Let $\sigma_1$ be a negative plurisubharmonic function given in the assumption and assume that $\{z:\sigma_1(z)<-c\}$ is relatively compact for some $c>0$, and contains $\zeta_0$. Choose a ball $B(\zeta_0,\varepsilon)$ in a coordinate neighborhood of $\zeta_0$ such that $$\overline{B(\zeta_0,\varepsilon)}\subset \{z:\sigma_1(z)<-c\}$$ and set $\sigma_2(z):=\sigma_1+\overline{c}$, where $-\overline{c} :=\max_{\xi\in \overline{B}(\zeta_0,\varepsilon)}\sigma_1(\xi)$. 

Fix a relatively compact, smooth strictly pseudoconvex domain $D\subset\subset M$ containing $\overline{\{z:\sigma_2(z)<\beta\}}$, $\beta:=\overline{c}-c$, and let $g_D$ denote the pluricomplex Green function on $D$ with a pole at $\zeta_0$. The function $g_D$ is continuous on $D$ and is a proper exhaustion on $D$ (see Theorem 4.3 of \cite{Demailly}, and Theorem 4.3 of \cite{P paper}). 

Choose $c_1>0$, $c_2>0$ so that;
\begin{eqnarray*}
\{z:g_D(z)<-c_1\}&\subset& B(\zeta_0,\varepsilon)\subset\subset \{z:\sigma_2(z)<\beta\}\\
&\subset\subset &\{z\in D:g_D(z)<-c_2\}\subset\subset D\subset\subset M.
\end{eqnarray*}

Set $\psi(z):= \big(\frac{c_1-c_2}{\beta}\big)\sigma_2-c_1$, and let
$$U:=\{z\in D:g_D(z)<-c_2\},\quad V:=\overline{\{z\in D:g_D<-c_1\}}^c\cap \{z\in D: g_D<-c_2\}.$$
For any $\xi\in \partial V\cap U$, $\overline{\lim}_{z\to \xi}\psi(z)\leq g_D(\xi)$ by construction. Hence the function $u$ defined by
\begin{equation*}
	u := \left\{
	\begin{array}{cl}
		\max(g_D,\psi) & \text{on } V,\\
		g_D & \text{on } U-V
	\end{array} \right.
\end{equation*}
is a plurisubharmonic function on $U$. 

Since on $\overline{\{z\in D:\sigma_2(z)<\beta\}}^c\cap\{z\in D:g_D<-c_2\}$, $\max\{g_D,\psi\}=\psi$, we can extend $u$ to be a bounded plurisubharmonic function on $M$ by setting it equal to $\psi$ off $\{z\in D:g_D(z)<-c_2\}$.

This modification makes $u-\sup_{z\in M}u(z)$ a semi-proper, negative plurisubharmonic function on $M$. Note that $u$
	is equal to $g_{D}$ in a neighborhood of $\zeta _{0}$. The 
domain $D$ is hyperconvex in the sense of Poletsky and so $g_{D}$,%
and hence $u$ has a strict logarithmic pole at $%
\zeta _{0}$ according to Theorem 4.3 of \cite{P paper}. Since $g_{M}(\cdot ,\zeta _{0})$ dominates  $%
u-\sup_{z\in M}u(z)$ on $M$, $g_{M}(\cdot ,\zeta _{0})$ also has a
strict logarithmic pole at $\zeta _{0}$.

\implication{(2)}{(3)} To simplify the notation, we will
	denote the pluricomplex Green function with a strict logarithmic pole at $%
\zeta _{0}$ by $p_{M}$. Pick a coordinate chart $U$%
 around $\zeta _{0}$, such that $g_{M}(\cdot ,\zeta
_{0})\geq \ln \parallel \cdot -\zeta _{0}\parallel +c$ for some $c$, on $U$. Choose an $\varepsilon >0$ such that the ball
	around $\zeta _{0}$ with radius $\varepsilon >0$\textbf{, }$%
B(\zeta _{0},\varepsilon )$, in this coordinate system, that
	satisfies $B(\zeta _{0},\varepsilon )\subset \subset U$. Now for $%
\beta \ll \ln \varepsilon -c$ small enough, we can find an $%
0<\varepsilon ^{-}<\varepsilon $ such that $\Omega _{\beta }\cap
B(\zeta _{0},\varepsilon )\subset B(\zeta _{0},\varepsilon ^{-})\subset
\subset B(\zeta _{0},\varepsilon )$, where $\Omega _{\beta
}:=\{z\in M:g_{M}(z,\zeta _{0})<\beta \}$. Since $\Omega _{\beta }$ 
 is connected, (\cite{PS}), we conclude that $\Omega _{\beta }\subset B(\zeta
_{0},\varepsilon )$ for this choice of $\beta$.  Hence $%
g_{M}(\cdot ,\zeta _{0})$ is a semi-proper negative
	plurisubharmonic function on $M$.

Choose $\varepsilon >0$, $c>0$, $0<c<c^{+}$%
 such that, 
  $$\overline{B(\zeta _{0},\varepsilon )}\subset
\{z:p_{M}(z)<-c^{+}\}\subset \{z:p_{M}(z)<-c\}\subset \subset M.$$ 
	Pick a $C^{\infty }$ strictly plurisubharmonic, proper function $%
\sigma :M\rightarrow \lbrack 0,\infty )$ and set $D_{k}=\{z:\sigma
(z)<k\}$, $k=1,2,\cdots $. For some $k_{1}$ large
	we have:
\begin{equation*}
\overline{B(\zeta _{0},\varepsilon )}\subset \{z:p_{M}(z)<-c^{+}\}\subset
\{z:p_{M}(z)<-c\}\subset \subset D_{k_{1}}.
\end{equation*}

Let $\Omega _{-}:=D_{k_{1}}$, $\Omega _{+}:=\overline{\{z\in M:p_{M}(z)<-c\}}%
^{c}$.

Define a plurisubharmonic function $\rho_t$ on $M$ via $\rho_t(z):=\frac{t}{c^+}p_M(z)+t$, $t>0$.

Clearly;
\begin{enumerate}
	\item $\rho_t(z)\leq 0$ for $z\in B(\zeta_0,\varepsilon)$, $t>0$,
	\item $\rho_t(z)\geq \big(1-\frac{c}{c^+}\big)t$ for $z\in \Omega_+$, $t>0$.
\end{enumerate}
We choose a Hermitian metric on $M$ and denote by $dV$ the volume form coming from this metric.  Using the notation of Lemma 1 of \cite{A manus} we set $d\varepsilon =cd\mu$, where $\mu$ is a measure on $M$ that is equivalent to the volume form, and $c$ is a strictly positive real valued function on $M$. 

Fix $f\in \mathcal{O}(\Omega_-)$ with $\big(\int_{D_{k_1}}|f|^2d\varepsilon\big)^{1/2}\leq 1$. 

In view of Lemma 1 of \cite{A manus} we can, for each $t>0$, compose $f$ as, $f=f_++f_-$; $f_+\in \mathcal{O}(\Omega_+)$, $f_-\in \mathcal{O}(\Omega_-)$ such that
\begin{eqnarray*}
	\int_{\Omega_+}|f_+|^2e^{-\rho_t}d\varepsilon&\leq &C_1e^{-(1-\frac{c}{c^+})t}\\
	\int_{\Omega_-}|f_-|^2e^{-\rho_t}d\varepsilon&\leq &C_1e^{-(1-\frac{c}{c^+})t}
\end{eqnarray*}
and for some $C_1>0$ which is independent of $f$ and $t$. 

Since $\exists\, C_2>0$ such that
$$\int_{\Omega_+\cap\Omega_-}|f|^2e^{-\rho_t}d\varepsilon\leq C_2e^{-(1-\frac{c}{c^+})t},$$
we have;
$$\int_{\Omega_\pm}|f_\pm|^2e^{-\rho_t}d\varepsilon\leq C_1e^{\frac{c}{c^+}t}$$
and 
$$\int_{B(\zeta_0,\varepsilon)}|f_-|^2d\varepsilon\leq \int_{B(\zeta_0,\varepsilon)}|f_-|^2e^{-\rho_t}d\varepsilon\leq C_1e^{-(1-\frac{c}{c^+}t)}.$$
Set
\begin{equation*}
	F := \left\{
	\begin{array}{ccl}
		f_+ & \text{on } \Omega_+,\\
		f-f_- & \text{on } \Omega_-.
	\end{array} \right.
\end{equation*}
The analytic function $F\in \mathcal{O}(M)$ satisfies
\begin{equation}\label{2star1}
\exists\, K>0;\quad \int_M|F|^2d\varepsilon \leq Ke^{t(\frac{c}{c^+})},
\end{equation}
$$\int_{B(\zeta_0,\varepsilon)}|F-f|^2d\varepsilon=\int_{B\zeta_0,\varepsilon)}|f_-|^2\leq c_1e^{-(1-\frac{c}{c^+})t}.$$
We set for $f\in \mathcal{O}(M)$;
$$\parallel f\parallel_k:=\big(\int_{D_k}|f|^2d\varepsilon\big)^{1/2},\quad k=1,2,\cdots.$$
Plainly, $\{\parallel\cdot\parallel_k\}_{k=1}^\infty$ forms a fundamental system of norms generating the topology of $\mathcal{O}(M)$, and as usual we denote the unit ball corresponding to $\parallel\cdot\parallel_k$ by $\mathcal{U}_k$, $k=1,2,\cdots$. 

Let $B:=\{g\in\mathcal{O}(M): \int_M|g|^2d\varepsilon\leq 1\}$. Observe that $B\subseteq \mathcal{U}_k$ for each $k\in \mathbb{N}$. 

The above analysis can be summerized as
\begin{equation}\label{2star2}
\exists\, k_1,\lambda>0;\, \forall\, k>k_1,\exists\, c>0;\, \mathcal{U}_{k_1}\subset \frac{1}{r^\lambda}\mathcal{U}_{B(\zeta_0,\varepsilon^-)}+cr\mathcal{U}_k,\,\forall\, r\geq 1,
\end{equation}
where $0<\varepsilon^-<\varepsilon$ and $\mathcal{U}_{B(\zeta_0,\varepsilon^-)}=\{f\in \mathcal{O}(M):\sup_{\xi\in B(\zeta_0,\varepsilon^-)}|f(\xi)|\leq 1\}$.

Since for $0<r\leq 1$, (\ref{2star2}) is trivial, we conclude that $M$ has the property $\widetilde{\Omega}_{\zeta_0}$.


\implication{(3)}{(1)}	Choose a volume form $dV$ on $M$ and an open\textit{\ strictly pseudoconvex }
exhaustion $\{D_{k}\}_{k=1}^{\infty }$, $\overline{D}_{k}\subset \mathring{D}%
_{k+1}$; $k=1,2,\cdots $, $\bigcup D_{k}=M$, of $M$ with $\zeta _{0}\in
D_{1} $. For $f\in \mathcal{O}(M)$, set $\parallel \cdot \parallel _{k}:=%
\big(\int_{D_{k}}|\cdot |^{2}dV\big)^{1/2}$, $k=1,2,\cdots $. In view of our
assumption $\exists $ $\varepsilon >0$ such that $B(\zeta _{0},\varepsilon
)\subset \subset D_{1}$ and \ref{2star1} 
\begin{equation*}
\exists\, k,\beta >0;\,\forall\, k>k_{1}\,\exists\, C_{k}>0\text{ such that }%
\mathcal{U}_{k_{1}}\subset r\mathcal{U}_{k}+\frac{C_{k}}{r^{\beta }}\mathcal{%
	U}_{0},\,\forall\, r>0, 
\end{equation*}%
where $\mathcal{U}_{k}$'s are closed unit balls corresponding to $\parallel
\cdot \parallel _{k}$'s and 
\begin{equation*}
\mathcal{U}_{0}=\{f\in \mathcal{O}(M):\sup_{\xi \in B(\zeta _{0},\varepsilon
	)}|f(\xi )|\leq 1\}. 
\end{equation*}

In the first part of the proof we employ the argument given in 29.16 Lemma
of \cite{MV} to construct a Hilbert space densely imbedded in $\mathcal{O}%
(M) $ and satisfies a certain interpolation condition. We will include this
line of reasoning here for completeness.

In view of nuclearity of $\mathcal{O}(M)$, for each $k=1,2,\cdots$, and $%
0<\varepsilon_k:=\min\{\frac{1}{2^k},\frac{1}{C_{k+2}^{\frac{2}{\beta}+1}}\}$%
, we can find a finite set $M_k\subset \mathcal{U}_{k+1}$ such that 
\begin{equation*}
\mathcal{U}_{k+1}\subset \varepsilon_k\mathcal{U}_k+M_k. 
\end{equation*}

Since for each $k=1,2,\cdots$, $M_s\subset\mathcal{U}_{k+1}$ for $s>k$, the
set $\Lambda=\bigcup_{s=1}^\infty M_s$ is bounded in $\mathcal{O}(M)$. For
each $g\in \mathcal{O}(M)^*$ we have: 
\begin{equation}  \label{2star3}
\parallel g\parallel^*_{k+1}\leq \varepsilon_k \parallel
g\parallel^*_k+\sup_{x\in \Lambda}|g(x)|,
\end{equation}
where $\parallel h\parallel^*_s:=\sup_{x\in \mathcal{U}_s}|h(x)|$ is the
dual norm of $\parallel\cdot\parallel_s$, $s=1,2,\cdots$.

Iterating inequality (\ref{2star3}) we get; 
\begin{align}  \label{1}
&\forall\, g\in\mathcal{O}(M)^*, \forall\, n>k_1 \\
\parallel g\parallel^*_n&\leq
(1+\varepsilon_{n-1}+\varepsilon_{n-1}\varepsilon_{n-2}+\cdots+%
\varepsilon_{n-1}\cdots\varepsilon_0)\sup_{x\in
	\Lambda}|g(x)|+\varepsilon_0\cdots\varepsilon_{n-1}\parallel g\parallel^*_0 
\notag \\
&\leq 2\sup_{x\in
	\Lambda}|g(x)|+\varepsilon_0\cdots\varepsilon_{n-1}\parallel g\parallel^*_0,
\notag
\end{align}
where $\parallel g\parallel_0^* =\sup_{x\in\mathcal{U}_0}|g(x)|$, and $%
\varepsilon_0=1$.

From (\ref{2star2}) it follows that, $\exists\, k_1,\beta>0$; $\forall\, n>k_1\,
\exists\, C_n$ such that $\forall\, g\in \mathcal{O}(M)^*$; 
\begin{equation*}
\parallel g\parallel^*_{k_1}\leq r\parallel g\parallel^*_n+ \frac{C_n}{%
	r^\beta}\parallel g\parallel^*_0,\quad \forall\, r>0. 
\end{equation*}
Inserting (\ref{1}) into this expression one gets: 
\begin{equation}  \label{2}
\parallel g\parallel^*_{k_1} \leq 2r \sup_{x\in \Lambda}|g(x)|+
(r\varepsilon_0\cdots \varepsilon_{n-1}+\frac{C_n}{r^\beta})\parallel
g\parallel^*_0,\quad \forall\, r, \forall\, g\in\mathcal{O}(M)^*.
\end{equation}
Now for $C_n\leq r^{\beta/2} \leq C_{n+1}$; 
\begin{equation*}
r\varepsilon_0\cdots \varepsilon_{n-1}\leq
C_{n+1}^{2/\beta}\varepsilon_0\cdots \varepsilon_{n-1}\leq \frac{%
	C_{n+1}^{2/\beta}}{C_{n+1}^{1/\beta+1}}=\frac{1}{C_{n+1}}\leq \frac{1}{%
	r^{\beta/2}} 
\end{equation*}
\begin{equation*}
\frac{C_n}{r^\beta}\leq \frac{r^{\beta/2}}{r^\beta}\leq \frac{1}{r^{\beta/2}}%
. 
\end{equation*}
Hence $\exists$ $R_0$; 
\begin{equation}  \label{3}
\parallel g\parallel^*_{k_1}\leq 2r\sup_{x\in \Lambda}|g(x)|+\frac{2}{%
	r^{\beta/2}}\parallel g\parallel^*_0,\quad \forall\, r\geq R_0.
\end{equation}
To get a Hilbert ball from $\Lambda$ we first consider the closed absolutely
convex hull of $\Lambda$; call it $\hat{\Lambda}$, and consider the Hilbert
norm $|\cdot|^2:= \sum_{k=1}^\infty \parallel\cdot \parallel^2_k\frac{1}{%
	2^kD_k^2}$, where $D_k=\sup_{\xi\in \Lambda}\parallel \xi\parallel_k$, $%
k=1,2,\cdots$, and let $B:=\{f\in\mathcal{O}(M):|f|\leq 1\}$. We denote by $%
H_B$ the Hilbert space imbedded in $\mathcal{O}(M)$, whose unit ball is $B$.

Taking into account that $\mathcal{O}(M)$ is separable, by enlarging $\Lambda$
if necessary, we can, without loss of generality, assume that $H_B$ is dense in 
$\mathcal{O}(M)$.

In view of (\ref{3}); 
\begin{equation}  \label{4}
\exists\, D>0;\, \forall\, g\in \mathcal{O}(M)^*;\quad \parallel
g\parallel^*_{k_1}\leq r\parallel g\parallel^*_B+\frac{D}{r^{\beta/2}}%
\parallel g\parallel^*_0;\quad \forall\, r>0,
\end{equation}
where $\parallel g\parallel^*_B:=\sup_{x\in B}|g(x)|$.

Taking minimum (over $r$) of the expression (\ref{4}) we obtain an
equivalent condition to $\widetilde{\Omega }_{\zeta _{0}}$, namely: 
\begin{align}
\exists\, \varepsilon >0,k_{1},& \,0<\gamma <1,\text{ and a Hilbert space }%
H_{\infty }\text{ densely imbedded in }\mathcal{O}(M)  \label{5} \\
\text{ and }c>0:&  \notag \\
\parallel g\parallel _{k_{1}}^{\ast }& \leq c\big(\parallel g\parallel
_{0}^{\ast }\big)^{\gamma }\big(\parallel g\parallel _{B}^{\ast }\big)%
^{1-\gamma }\quad \forall\, g\in \mathcal{O}(M)^{\ast },  \notag
\end{align}%
\textit{where} $\parallel \cdot \parallel _{0}^{\ast }=\sup_{\mathcal{U}%
	_{0}}|\cdot |$ .

Let $H$ denote closure of $\mathcal{O}(M)$ with respect to the
hilbertian norm $\left\Vert \cdot \right\Vert =\left( \int\limits_{B(\zeta
	_{0},\varepsilon )}\left\vert \cdot \right\vert ^{2}dV\right) ^{\frac{1}{2}%
}. $ Since for some constant $c>0,$ $\parallel g\parallel _{0}^{\ast }\leq
c\parallel g\parallel _{{}}^{\ast }\forall\, g\in \mathcal{O}(M)^{\ast },$ the
above statement holds for the dual norm $\parallel \cdot \parallel
_{{}}^{\ast }$as well.

The imbedding $\imath :H_{\infty }\hookrightarrow H_{{}}$ is
compact since $\mathcal{O}(M)$ is nuclear, and consequently has a
representation of the form: 
\begin{equation*}
\imath (x)=\sum_{n=0}^{\infty }\lambda _{n}\langle x,f_{n}\rangle _{\infty
}e_{n}, 
\end{equation*}%
for an orthonormal basis $\{f_{n}\}$, an orthonormal system $\{e_{n}\}$ for $%
H_{\infty }$ and $H_{{}}$ respectively, for some null sequence $%
\{\lambda _{n}\}$ consisting of strictly positive numbers.

Note that $\parallel f_{n}\parallel =\lambda _{n}$ and $\parallel
f_{n}^{\ast }\parallel _{{}}^{\ast }=\frac{1}{\lambda _{n}}$, where $%
f_{n}^{\ast }(\cdot )=\langle \cdot ,f_{n}\rangle _{\infty }$, $n=0,1,\cdots 
$.

Set $\varepsilon _{n}:=\ln \big(\frac{1}{\lambda _{n}}\big)$. Since $%
\{f_{n}\}$ is a bounded sequence in $\mathcal{O}(M)$, and $\varepsilon
_{n}\rightarrow \infty $, the function 
\begin{equation*}
\psi (z):=\overline{\lim }_{\xi \rightarrow z}\overline{\lim }_{n}\frac{\ln
	|f_{n}(\xi )|}{\varepsilon _{n}} 
\end{equation*}%
defines a plurisubharmonic function on $M$. Plainly, $\psi (z)\leq 0$ for $%
z\in M$, and since $\parallel f_{n}\parallel =\lambda _{n}$ for $z\in
B(\zeta _{0};\varepsilon ^{-})$ , $\varepsilon ^{-}<\varepsilon $, $\psi
(z)\leq -1$. Choose a $0<\gamma ^{+}<1$ with $\gamma ^{+}>\gamma $.

We will look at the sublevel set 
\begin{equation*}
\Omega_{\gamma^+}=\{z:\psi(z)<-\gamma^+\}. 
\end{equation*}
This set is non-empty since $B(\zeta_0;\varepsilon)\subseteq
\Omega_{\gamma^+}$.

Choose a point $^{{}}z\in \Omega _{\gamma ^{+}}$. In view of
Hartog's Lemma (p.50 in \cite{K}), $\exists\, c_{{}}=c(z) >0 $
such that 
\begin{align}
|f_{k}(z)|& \leq ce^{-\gamma ^{+}\varepsilon _{k}},\,k=1,2,\cdots ,\quad
\label{6} \\
\text{For an }f\text{ }& \in H_{\infty }\text{,\ \ \ \ \ \ \ \ \ \
	\ \ }f\text{ }(z)=\sum_{n}f_{n}^{\ast }(f)f_{n}(z),  \notag
\end{align}%
in view of (\ref{5}) and (\ref{6}) we have; 
\begin{eqnarray*}
	|f(z)| &\leq &\sum_{n}|f_{n}^{\ast }(f)||f_{n}(z)|\leq \big(%
	\sum_{n}\parallel f_{n}^{\ast }\parallel _{k_{1}}^{\ast }|f_{n}(z)|\big)%
	\parallel f\parallel _{k_{1}} \\
	&\leq &\tilde{c}\big(\sum_{n}\frac{1}{\lambda _{n}^{\gamma }}e^{-\gamma
		^{+}\varepsilon _{n}}\big)\parallel f\parallel _{k_{1}} \\
	&=&\tilde{c}\big(\sum_{n}e^{(\gamma -\gamma ^{+})\varepsilon _{n}}\big)%
	\parallel f\parallel _{k_{1}}\leq \kappa \parallel f\parallel _{k_{1}}\text{
		\ for some \ }\kappa =\kappa \left( z\right) . \\
	&&\text{\ \ \ \ \ \ \ \ \ \ \ \ \ \ \ \ \ \ \ \ \ \ \ \ \ \ \ \ \ \ \ \ \ \
		\ \ \ \ \ \ \ \ \ \ \ \ \ \ \ \ \ \ \ \ \ \ \ \ \ \ \ \ \ \ \ \ \ \ \ \ \ \
		\ \ \ \ \ \ \ \ \ \ \ \ \ \ \ \ \ \ \ \ \ \ \ \ \ \ \ \ }
\end{eqnarray*}

In the above estimate we have used the fact that increasing
permutation of the sequence $\left( \varepsilon _{n}\right) _{n}$ dominates
by a constant times the sequence $\left( \frac{1}{d_{n}\left(
	U,U_{k_{1}}\right) }\right)_{n}$, $U$ being the unit ball of $H$, which in turn dominates a constant
times $\left( n^{\frac{1}{d}}\right) _{n}$, where $d=dimM$ (See
Proposition 1.1 of \cite{AKT manuscripta}). Since $H_{\infty }$ is dense in $%
\mathcal{O}(M)$, the above inequality is true for each $f\in \mathcal{O}(M)$%
. Choose a relatively compact holomorphically convex $\hat{D}_{k_{1}}\subset
\subset M$ that contains $\overline{D}_{k_{1}}$. The analysis above yields: 
\begin{equation*}
\exists\, C>0:\forall\, f\in \mathcal{O}(M)\quad |f(z)|\leq C\sup_{\xi \in \hat{D%
	}_{k_{1}}}|f(\xi )|,\quad j=1,2,\cdots . 
\end{equation*}%
Employing a well known trick of applying this inequality to $f^{n}$
for a given $f\in \mathcal{O}(M)$, taking $n$th root of both sides as $%
n\rightarrow \infty $ we conclude that $z\in \hat{D}_{k_{1}}$.

This finishes the proof of the theorem.

\end{proof}

We now look into the question posed in the beginning of the section. 

\begin{theorem}\label{thm3}
	Let $M$ be a Stein manifold of dimension $d$, and suppose it satisfies one of the equivalent conditions of Theorem \ref{thm2} for some $\zeta_0\in M$. Then $\Delta(\mathcal{O}(M))=\Delta(\mathcal{O}(\Delta^d))$. 
\end{theorem}

\begin{proof}
	Choose a $\mathcal{C}^\infty$-strictly plurisubharmonic exhaustion function $p:M\to (-\infty,\infty)$ and let $\Omega_k=\{z:p<k\}$, $k\in \mathbb{Z}^+$ with $p$ choosen such that $\zeta_0\in \Omega_0$. We assume that $M$ has the property $\hat{\Omega}_{\zeta_0}$. Choose $\varepsilon<\varepsilon^+$, $\overline{B(\zeta_0,\varepsilon^+)}\subset \Omega_{k_1}$ and a sequence of integers $\{k_n\}$ with $k_1<k_2<\cdots$ tending to infinity, where $\varepsilon$ and  $k_1$ are as in the definition of $\widetilde{\Omega}_{\zeta_0}$. Using the notation of Theorem \ref{thm2}, \implication{(3)}{(1)}, we set $$\mathcal{V}_0=\{f\in \mathcal{O}(M):\big(\int_{B(\zeta_0,\varepsilon^+)}|f|^2d\varepsilon\big)^{1/2}\leq 1\}\quad \textrm{and}$$ 
	$$\mathcal{V}_n=\{f\in \mathcal{O}(M):\big(\int_{D_{k_n}}|f|^2d\varepsilon\big)^{1/2}\leq 1\},\quad n=1,2,\cdots.$$ 
	In view of Proposition 3.8 of \cite{AS} and the condition $\widetilde{\Omega}_{\zeta_0}$ we have: 
	
	\begin{equation}\label{3star1}
	\forall\, n=0,1,\cdots, \forall\, k>n, \, \exists\, j\textrm{ and }C>0\textrm{ with } \mathcal{V}_{k+1}\subset\subset r^j\mathcal{V}_n+\frac{C}{r}\mathcal{V}_k,\quad \forall\, r>0.
	\end{equation}
	Also we have:
	\begin{equation}\label{3star2}
	\forall\, k=0,\cdots,\exists\, 0<\gamma<1 \textrm{ and } C>0:\,\parallel f\parallel_{k+1}\leq C \parallel f\parallel_k^\gamma \parallel f\parallel_{k+2}^{1-\gamma}.
	\end{equation}
	This "two-constants theorem" type condition can be perceived in many ways, one of which is to employ $p$-measures in its proof as in \cite{Ay istanbul} p. 125. Now, from the discussion leading to Proposition 1.1 of \cite{AKT manuscripta} we conclude that:
	$$0<\underline{\lim}\frac{-\ln{d_n(\mathcal{V}_1;\mathcal{V}_0)}}{\alpha_n}\leq \overline{\lim}\frac{-\ln{d_n(\mathcal{V}_1;\mathcal{V}_0)}}{\alpha_n}<\infty,$$
	where $\alpha_n:= n^{1/d}$, $d=\dim M$, is the \emph{associated exponent} sequence of $\mathcal{O}(M)$ (\cite{AKT manuscripta}). Choose $R_0$ large, $1\ll R_0$ such that:
	\begin{equation}\label{3star3}
	\exists\, n_0:\, n\geq n_0 \, \Rightarrow \, \frac{1}{R_0^{\alpha_n}}\leq d_n(\mathcal{V}_1;\mathcal{V}_0).
	\end{equation}
	We estimate $d_n(\mathcal{V}_1;\mathcal{V}_0)$, $n\in \mathbb{N}$ using $\widetilde{\Omega}_{\zeta_0}$ as in \cite{T Karadeniz}. In view of 29.13 Lemma of \cite{MV} and (\ref{3}) of the proof \implication{(2)}{(3)} of Theorem \ref{thm2}, an equivalent form of $\widetilde{\Omega}_{\zeta_0}$ in our notation reads as; $\exists$ bounded set $B$ and $\lambda>0$ such that 
	\begin{equation}\label{3star4}
	\mathcal{V}_1\subseteq \frac{1}{r^\lambda}\mathcal{V}_0+CrB;\quad \forall\, r>0.
	\end{equation}
	
	Now suppose that $\rho>0$ satisfies $B\subseteq \rho\mathcal{V}_0+\mathcal{L}$, for some finite dimensional subspace $\mathcal{L}$ with $\dim(\mathcal{L})\leq n$. 
	
	In view of (\ref{3star4}), this leads to 
	$$\mathcal{V}_1\subseteq \big(\frac{1}{r^\lambda}+Cr\rho\big)\mathcal{V}_0+\mathcal{L},\quad \forall\, r>0,$$
	and hence to:
	$$d_n(\mathcal{V}_1;\mathcal{V}_0)\leq \big(\frac{1}{r^\lambda}+Cr\rho\big) \mathcal{V}_0+\mathcal{L},\quad \forall\, r>0.$$
	
	Taking infimum of $\big(\frac{1}{r^\lambda}+Cr\rho\big)$ over $r$ first, then taking infimum over $\rho$ we get:
	
	\begin{equation}\label{3star5}
	\exists\, \kappa >0; \quad d_n(\mathcal{V}_1;\mathcal{V}_0)\leq \kappa d_n(B,\mathcal{V}_0)^{\lambda/1+\lambda}.
	\end{equation}
	
	We claim that 
	$$\Delta(\mathcal{O}(M))\neq \Delta(\mathcal{O}(\mathbb{C}^d))=\{(\xi_n):\exists\, R\geq 1\textrm{ and }C>0; |\xi_n|\leq CR^{\alpha_n}\}.$$
	
	Anticipating a contradiction, we assume that $\Delta(\mathcal{O}(M))= \Delta(\mathcal{O}(\mathbb{C}^d))$. Choose an $R\gg R_0^{1+1/\lambda}$. Since $\{R^{\alpha_n}\}_n\in \Delta(\mathcal{O}(\mathbb{C}^d))$, $\exists\, m\in \mathbb{N}$ such that
	\begin{equation}\label{3star6}
	\lim_nR^{\alpha_n}d_n(\mathcal{V}_m;\mathcal{V}_0)=0.
	\end{equation}
	
	So by (\ref{3star3}) and (\ref{3star5}):
	\begin{eqnarray*}
	\Big(\frac{R}{R_0^{1+1/\lambda}}\Big)^{\alpha_n}&\leq& R^{\alpha_n}d_n(\mathcal{V}_1;\mathcal{V}_0)^{1+1/\lambda}\leq \big(\kappa^{1+1/\lambda}\big)R^{\alpha_n}d_n(B;\mathcal{V}_0)\\
	&\leq &CR^{\alpha_n}d_n(\mathcal{V}_m;\mathcal{V}_0).
	\end{eqnarray*}

	Hence, (\ref{3star6}) gives us the desired contradiction. 
	
	In view of Theorem 4.4 of \cite{Ay} we conclude that $\Delta(\mathcal{O}(M))=\Delta(\mathcal{O}(\Delta^d))$. This finishes the proof of the theorem. 
\end{proof}

We now look at some immediate corollaries of Theorem \ref{thm3}.

\begin{corollary}\label{cor4}
	Let $M$ be an open Riemann surface. Then the following are equivalent:
	\begin{enumerate}
		\item $M$ is hyperbolic.
		\item $M$ possesses a semi-proper negative subharmonic function.
		\item $\Delta(\mathcal{O}(M))=\Delta(\mathcal{O}(B(0,1)))=\{(\xi_n):\sup_n|\xi_n|r^n<+\infty,\,\forall\, r<1\}$. 
	\end{enumerate}
\end{corollary}

\begin{proof}
	If $M$ is hyperbolic, it has a non-trivial Green's function $g_M(\cdot,\cdot)$ at each point of $M$, which is harmonic except at its singularity. Fix a point $\zeta\in M$ and a small ball $B(\zeta,\varepsilon)$ in some chart around $\zeta$.
	
	Consider the sublevel set $\{z: g_M(z;\zeta)<-c\}$ where $-c<\min_{\varepsilon \in \partial B(\zeta,\varepsilon)}g_M(\zeta,\varepsilon)$. Such a $c>0$ exists since $g_M(\zeta,\cdot)$ is continuous. Taking into account that the sublevel sets of the Green functions are connected (\cite{PS}), we conclude that $$\{z:g_M(z,\zeta)<-c\}\subset B(\zeta,\varepsilon).$$ 
	
	So $M$ has a semi-proper negative subharmonic function. The other implications are consequences of Theorem \ref{thm3} and Theorem \ref{thm1}, respectively.
\end{proof}

Another consequence of Theorem \ref{thm3} and Theorem \ref{thm1} is:

\begin{corollary}\label{cor5}
	Let $M$ be a Stein manifold. If $M$ possesses a semi-proper complex Green function with a pole at $\zeta_0\in M$, then for each $\zeta\in M$ pluricomplex Green function with a pole at $\zeta$ exists. 
\end{corollary}

We would like to close this section by recalling a linear topological invariant introduced by D. Vogt (\cite{VB}) from the structure theory of nuclear Fr\'echet spaces. 

\begin{definition}
	Let $\mathcal{X}$ be a Fr\'echet space and $\{\parallel \cdot\parallel_k\}_k$ a fundamental system of seminorms generating the topology. $\mathcal{X}$ is said to have the property $\widetilde{\Omega}$ in case;
	$$\forall\, k_0\,\exists\, k_1,\lambda >0;\,\forall\, k>k_1, \,\exists\, C>0,\quad \mathcal{U}_{k_1}\subseteq \frac{C}{r}\mathcal{U}_{k_0}+r^\lambda \mathcal{U}_k,\quad \forall\, r>0,$$
	where as usual, $\mathcal{U}_s$ denotes the closed unit ball of the local Banach space $\mathcal{X}_s$, $s=1,2,\cdots$.
\end{definition}

This property does not depend upon the generating system $\{\parallel\cdot\parallel_k\}$ and is in fact a linear topological invariant; that is, if $\mathcal{X}$ has $\widetilde{\Omega}$ then any Fr\'echet space isomorphic to $\mathcal{X}$ as a linear topological space, also enjoys this property. 

Like all $\Omega$-type conditions, $\widetilde{\Omega}$ passes to quotient spaces. For more information about the property $\widetilde{\Omega}$, we refer the reader to (\cite{VB,DMV Bulletin France}). Stein manifolds $M$, for which $\mathcal{O}(M)$ possesses the property $\widetilde{\Omega}$ were characterized in \cite{A D-issue}. 

Plainly the property $\widetilde{\Omega}$ for $\mathcal{O}(M)$, $M$ a Stein manifold, implies that $M$ possesses the property $\widetilde{\Omega}_\zeta$ for every point $\zeta$ is $M$. Consequently, $\widetilde{\Omega}$ appears as a linear topological invariant of $\mathcal{O}(M)$ that ensures for every $\zeta\in M$ the existence of a semi-proper negative plurisubharmonic on $M$ with a relatively compact sublevel set containing $\zeta$. 

On the other hand, possessing $\widetilde{\Omega}_\zeta$ at every point for a Stein manifold $M$ need not imply that $\mathcal{O}(M)$ has the property $\widetilde{\Omega}$. To see this, consider the punctured disc $B(0,1)\setminus\{0\}\subseteq\mathbb{C}$. This domain obviously enjoys the property $\widetilde{\Omega}_\zeta$ for each of its points. Employing Laurent expansion for analytic functions on the punctured disc, it is not difficult to show that $\mathcal{O}(B(0,1)\setminus\{0\})\approx \mathcal{O}(\mathbb{C})\times \mathcal{O}(B(0,1))$ as Fr\'echet spaces (see \cite{Rolewicz}, p. 376). If the space of analytic functions on this punctured disc had the property $\widetilde{\Omega}$, $\mathcal{O}(\mathbb{C})$, being a quotient of $\mathcal{O}(B(0,1)\setminus\{0\})$, would also enjoy this property. But this cannot be, eg., in view of Theorem  \ref{thm3} and the above remarks.

\section{Pluri-Greenian and Locally Uniformly Pluri-Greenian
	Complex Manifolds.}\label{section3}

E. Poletsky in $\left[ 18\right] $ introduced the concepts of \textit{%
	pluri-Greenian }and \textit{locally uniformly pluri-Greenian} complex
manifolds. In this section we look into these notions from our
	point of view.

A complex manifold $M$ is called \textit{pluri-Greenian} if for each $\omega
\in M$, the pluricomplex Green function $g_{M}(\cdot ,\omega )$ has a strict
logarithmic pole at $\omega $.

On pluri-Green manifolds the pluricomplex Green function $g_{M}(\cdot
,\omega )$ satisfies $(dd^{c}g_{M}(\cdot ,\omega ))^{d}=(2\pi )^{d}\delta
_{\omega }$ if $d$ is the dimension of the manifold.

Consequently, in view of Theorem 2.2, a Stein manifold $M$
	is pluri-Greenian if and only if it has the property $\widetilde{\Omega }%
_{z}$ for every $z\in M$.

Recall that a complex manifold $M$ is called \textit{locally} \textit{%
	uniformly} \textit{pluri-Greenian }in case every point $w_{0}\in M$ has
a coordinate neighborhood $U$ with the following property: there is an open
set $W\subseteq U$ containing $w_{0}$ and a constant $c$ such that $%
g_{M}\left( z,w\right) \geq \ln \left\Vert z-w\right\Vert +c$ on $U$
whenever $w\in W$ (see \cite{P paper}, \cite{P yeni}).

Balls in $\mathbb{C}^{N}$ plainly satisfy the required property; hence bounded domains in $\mathbb{C}^{N},$ are \textit{locally} \textit{uniformly} \textit{pluri-Greenian}. It turns out
that existence of a pluricomplex Green function with a strict logarithmic
pole at each point of a complex manifold ensures that it is locally
uniformly pluri-Greenian.

\begin{theorem}
	1) A complex manifold is locally uniformly pluri-Greenian if and only if it is
	pluri-Greenian.
	
	2) A Stein manifold $M$ is locally uniformly pluri-Greenian if and only if it has
	the property $\widetilde{\Omega}_{z}$ for each $z\in M$
\end{theorem}

\begin{proof}
	Let M be a complex manifold and $w_{0}\in M$. Suppose there exists a
	pluricomplex Green function $g_{M}\left( .,w_{0}\right) $ with a strict
	logarithmic pole at $w_{0}$. That is there exists numbers $A$ and $B$ such
	that%
	\[
	\ln \left\Vert z-w_0\right\Vert +A\leq g_{M}\left( z,w_{0}\right) \leq \ln
	\left\Vert z-w_0\right\Vert +B 
	\]%
	on a coordinate neighborhood $U_{0}$ around $w_{0}$. Choose $R>0$ and $\beta
	<<0$ for the time being, so that%
	\begin{eqnarray*}
		\lbrace z:z\in U_{0},\textrm{ }g_{M}\left( z,w_{0}\right) <\beta \rbrace
		&\subset &B\left( w_{0},e^{\beta -A}\right) \textrm{ and} \\
		B\left( w_{0},e^{\beta -A}\right) &\subset \subset &B\left( w_{0},R\right)
		\subset \subset U_{0}\textrm{ \ }
	\end{eqnarray*}%
	where $B\left( w,r\right) $ denotes the ball around $w\in U_{0}$ with
	radius $r>0$ in $U_{0}$. Since the sub-level sets of $g_{M}\left(
	.,w_{0}\right) $ are connected, 
	\[
	\Omega _{\beta }=\lbrace z:z\in M\textrm{ },\textrm{ }g_{M}\left(
	z,w_{0}\right) <\beta \rbrace \subset\subset B\left( w_{0},R\right) \subset \subset
	U_{0}\textrm{ } 
	\]%
	Choose $R^{-}<<R$ such that $B\left( w_{0},R^{-}\right) \subset \subset $ $%
	\Omega _{\beta }$, and set $\beta _{0}=\sup_{\zeta \in B\left(
		w_{0},R^{-}\right) }g_{M}\left( \zeta ,w_{0}\right)$. Set,%
	\[
	\sigma \left( z\right) =g_{M}\left( z,w_{0}\right) -\beta _{0},\ \ z\in
	M. 
	\]

The plurisubharmonic function $\sigma $ is bounded and semi-proper, since%
\[
\lbrace z:\sigma \left( z\right) <\beta -\beta _{0}\rbrace \subset \subset
B\left( w_{0},R\right) .\textrm{ \ } 
\]

Let $g$ denote the pluricomplex Green function of $B\left( w_{0},R\right) $
with a pole at $w_{0}$. Choose constants $r_{1}$ and $r_{2}$ , $r_{1}<r_{2},$
such that ;%
\begin{eqnarray*}
	\lbrace z\in B\left( w_{0},R\right) :g\left( z\right) <r_{1}\rbrace
	&\subset \subset &B\left( w_{0},R^{-}\right) \\
	B\left( w_{0},e^{\beta -A}\right) &\subset \subset& \lbrace z\in B\left(
	w_{0},R\right) :g\left( z\right) <r_{2}\rbrace \subset \subset B\left(
	w_{0},R\right) .\textrm{ \ \ \ }
\end{eqnarray*}

Set%
\[
\Phi \left( z\right) =\frac{r_{2}-r_{1}}{\beta -\beta _{0}}\sigma \left(
z\right) +r_{1},\textrm{ }z\in M\textrm{ and }\Omega =\lbrace z\in
B\left( w_{0},R\right) :g\left( z\right) <r_{2}\rbrace . 
\]

Define a function $u$ on $\Omega $ , by the formula: 
\begin{equation*}
	u\left( z\right) = \left\{
	\begin{array}{cl}
		\max \left( g,\Phi \right)& \textrm{on }\overline{\lbrace z\in
			B\left( w_{0},R\right) :\textrm{ }g\left( z\right) <r_{1}\rbrace }^{c}\cap
		\Omega \\
		g\left( z\right) &\textrm{on }%
		\lbrace z\in B\left( w_{0},R\right) :g\left( z\right) \leq r_{1}\rbrace.
	\end{array} \right.
\end{equation*}
By our construction $u$ becomes a plurisubharmonic function. Note that
outside $B\left( w_{0},e^{\beta -A}\right) $, $g_{M}\left( z,w_{0}\right)
\geq \beta ,$ so $\sigma \geq \beta -\beta _{0}$ outside $B\left(
w_{0},e^{\beta -A}\right) \cap \Omega$. Hence by setting $u=\Phi $ outside $%
\Omega $, we can extend $u$ to $M$ to a bounded plurisubharmonic function on 
$M$. Note that on $M,$ 
\[
u\left( z\right) \leq \sup_{\zeta \in M}\Phi \left( \zeta \right) \leq 
\frac{r_{2}-r_{1}}{\beta -\beta _{0}}\left( -\beta _{0}\right) +r_{1} 
\]

Setting $u_{M}=\frac{r_{2}-r_{1}}{\beta -\beta _{0}}\left( -\beta
_{0}\right) +r_{1},$ the function $v=u-u_{M}$ becomes a negative
plurisubharmonic function on M that is equal to $g-u_{M}$ on a ball around $%
w_{o}$ where $g$ is the pluricomplex Green function of a ball of radius $R$
with center $w_{0}$ in a coordinate chart around $w_{o}.$


Note that the dependence of $\Phi $ to the function $g$ is through the
constants $r_{1}$ and $r_{2}$ that satisfy the conditions above. If one could
select $r_{1}$ and $r_{2}$ such that the conditions required by these
numbers are fulfilled by all pluricomplex Green functions of $B\left(
w_{0},R\right) $ with poles that lie in a fixed small ball around $w_{0}$;
one can then use the function $\Phi $ defined above and carry out the above
argument simultaneously for all pluricomplex Green functions with poles in
this neighborhood and thereby obtain a negative plurisubharmonic function on 
$M,$ for each $x$ in this neighborhood, which equals to $g_{x}+C$, for some $%
C $ that does not depend upon $x,$ on a fixed neighborhood $W$ of $w_{0}$,
 where $g_{x}$ denotes the pluricomplex Green function of $B\left(
w_{0},R\right) $ with a pole at $x.$ We claim that such a choice of
constants is possible.

In other words we wish to show that;
$$\exists \, s_{1}\textrm{ and }s_{2}, \quad s_{1}<s_{2}<0\textrm{ and }\delta >0: $$%
$$B\left( w_{0},\delta \right) \subset \lbrace z\in U_{0}:g_{x}\left(
z\right) <s_{1}\rbrace \subset \subset B\left( w_{0},R^{-}\right) \textrm{ and }$$
$$B\left( w_{0},e^{\beta -A}\right) \subset \subset \lbrace z\in B\left(
w_{0},R\right) :g_{x}\left( z\right) <s_{2}\rbrace \subset \subset B\left(
w_{0},R\right),$$
 for every $x$ in a fixed neighborhood of $w_{0}$.


To this end, choose $r_{1}$ and $r_{2}$ such that the condition is satisfied
for $x=w_{0}$. In view of Lemma 6.2.4 of \cite{K} choose balls around $w_0$, $B$ and $B_{0}$ such that $w_{0}\in B\subset \subset
B_{0}\subset \subset \lbrace z\in B\left( w_{0},R\right)
:g_{w_{0}}\left( z\right) <r_{1}\rbrace ,$%
\[
\left( 1+\varepsilon \right) g_{w_{0}}\left( z\right) \leq g_{x}\left(
z\right) \leq \frac{1}{1+\varepsilon }g_{w_{0}}\left( z\right) ,\ \ \
x\in B,\textrm{ }z\in U_{0}-B_{0},\ \ 
\]

and $\varepsilon >0$ to be chosen later.

On $g_{w_{0}}=r_{1}$, $g_{x}\leq \frac{r_{1}}{1+\varepsilon }$, hence $%
g_{x}\left( z\right) <\frac{r_{1}}{1+\varepsilon }$ , on $\left(
z:g_{w_{0}}<r_{1}\right) $ for all $x\in B$. Consequently we can find a $%
\delta >0,$ such that%
\[
B\left( w_{0},\delta \right) \subset \lbrace z\in
U_{0}:g_{w_{0}}<r_{1}\rbrace \subset \lbrace z\in U_{0}:g_{x}\left(
z\right) <\frac{r_{1}}{1+\varepsilon }\rbrace 
\]

for all $x\in B.$ On the other hand for $z\notin B_{0},$ and $x\in
B,$ one has: 
$$g_{x}\left( z\right) <\frac{r_{1}}{1+\varepsilon }%
\Longrightarrow g_{w_{0}}\left( z\right) <\frac{r_{1}}{\left( 1+\varepsilon
	\right) ^{2}}.$$ 
Choose $\varepsilon >0$ so small that $\lbrace z:g_{w_{0}}\left( z\right) <\frac{r_{1}}{\left(
	1+\varepsilon \right) ^{2}}\rbrace \subset \subset B\left( w_{0},R^{-}\right) .$
It follows that $\lbrace z\in B\left( w_{0},R\right) :g_{x}\left(
z\right) <\frac{r_{1}}{1+\varepsilon }\rbrace \subset \subset B\left(
w_{0},R^{-}\right) $, $\forall\, x\in B.$

Choose a ball $\Delta $ around $w_{0}$ that is contained compactly in $%
\lbrace z\in B\left( w_{0},R\right) :g_{w_{0}}\left( z\right)
<r_{2}\rbrace $ and contains $\overline{B\left( w_{0},e^{\beta -A}\right) }$.

On the boundary of $\Delta ,$ $g_{x}\leq \frac{r_{2}}{1+\varepsilon }$\ \ \ for 
$x\in B$. So; 
\[
B\left( w_{0},e^{\beta -A}\right) \subset \subset \lbrace z\in B\left(
w_{0},R\right) :g_{x}\left( z\right) <\frac{r_{2}}{1+\varepsilon }\rbrace . 
\]

We put another condition on $\varepsilon $, by requiring that 
$$\lbrace z\in B\left( w_{0},R\right) :g_{w_{0}}\left( z\right) <\frac{r_{2}}{\left(
	1+\varepsilon \right) ^{2}}\rbrace \subset \subset B\left( w_{0},R\right) .$$
It now follows that;%
\[
\lbrace z\in B\left( w_{0},R\right) :g_{x}\left( z\right) <\frac{r_{2}}{%
	1+\varepsilon }\rbrace \subset \subset B\left( w_{0},R\right) 
\]

By choosing $\varepsilon >0$ satisfying the requirements stated above and
setting $s_{1}=\frac{r_{1}}{1+\varepsilon },$ $s_{2}=\frac{r_{2}}{1+\varepsilon }, 
$ one proves the claim.

Now since $B\left( w_{0},R\right) $ is locally uniformly pluri-Greenian,
there are open neighborhoods $W$ and $U$ of $w_{0}$, $W\subseteq U,$
contained in $B$ and a constant $c$ such that $g_{x}\left( z\right) \geq
\ln \left\Vert z-x\right\Vert +c$ on $U$ whenever $x\in W$. Combining
this with the analysis above finishes the proof of the first part of the
theorem.

The second part of the theorem follows directly from the first part and Theorem
2.2.
\end{proof}

\section{Concluding Remarks}

\subsection*{\RomanNumeralCaps{1}} We would like to point out that the results  about locally
uniformly pluri-Greenian complex manifolds obtained by E. Poletsky in his
 papers \cite{P paper}, \cite{P yeni} are valid, in view of
Theorem 3.1, for pluri-Greenian complex manifolds. This, in particular
applies to the construction in \cite{P yeni} of \textit{%
	pluri-potential compactifications }for locally uniformly pluri-Greenian
complex manifolds. Poletsky introduced biholomorphically invariant \textit{\
	pluri-potential compactifications }for complex manifolds in \cite{P yeni} as a counterpart to the Martin compactification in the classical
potential theory. Consequently Theorem 3.1 and Theorem 2.1 implies:

\begin{corollary}
	Let $M$ be a Stein manifold and suppose the Fr\'{e}chet space $O\left(
	M\right) $ has the property $\widetilde{\Omega}$. Then $M$ admits a
	pluri-potential compactification in the sense of Poletsky.
\end{corollary}

\subsection*{\RomanNumeralCaps{2}} \ In \cite{HSTI}, Harz, Shcherbina and Tomassini introduced the notion of a 
\emph{core} of a complex manifold $M$, $c(M)$, as the set of all points $%
w\in M$ where every smooth global plurisubharmonic function that is bounded
from above fails to be strictly plurisubharmonic near $w$.

In the example of Section \ref{section1}, the core of the manifold is
plainly itself.

Poletsky and Shcherbina in \cite{PS} Corollary 3.3 showed that complex
manifolds possess pluricomplex Green functions with strict logarithmic
singularities at points outside their core. Combining this result with
Theorem \ref{thm2} of Section \ref{section2}, we get;

\begin{corollary}
	A Stein manifold $M$ satisfies $\widetilde{\Omega}_\zeta$, for all points $\zeta\in M$ that lie outside its core. 
\end{corollary}

\subsection*{\RomanNumeralCaps{3}} \ Chen and Zhang, in their paper \cite{CZ}, considered Stein manifolds whose pluricomplex Green functions are semi-proper and specified them as manifolds that satisfy the property $(B1)$. They showed, among other things, that Stein manifolds satisfying property $(B1)$ possesses a Bergman metric (\cite{CZ} Theorem 1). This result and Theorem \ref{thm2} of Section \ref{section2} yields:

\begin{corollary}
	A Stein manifold $M$ satisfying $\widetilde{\Omega}_\zeta$ for all its points, possesses a Bergman metric.
\end{corollary}

\section*{Acknowledgement}

The author sincerely thanks Proffessor E. Poletsky for his valuable comments.

\bibliographystyle{amsplain}

\end{document}